\documentclass[12pt]{amsart}
\usepackage[pdftex]{hyperref}
\usepackage{amssymb}
\usepackage{multirow}
\usepackage{tikz}
\usetikzlibrary{cd,matrix,arrows,decorations.pathmorphing}

\def\Hbar{\bar{H}}
\def\Abar{\bar{A}}

\allowdisplaybreaks
\newtheorem{theorem}{Theorem}
\newtheorem{lemma}[theorem]{Lemma}
\newtheorem{corollary}[theorem]{Corollary}
\theoremstyle{remark}
\newtheorem{example}[theorem]{Example}
\newtheorem{examples}[theorem]{Examples}
\renewcommand{\theenumi}{\roman{enumi}}
\numberwithin{theorem}{section}
\numberwithin{equation}{section}

\def\set#1#2{\{#1:#2\}}

\def\gend#1{\langle #1\rangle}
\def\biggend#1{\bigl\langle#1\bigr\rangle}
\let\iso\cong
\def\sset{\subseteq}
\def\subgp{\leq}
\def\normal{\trianglelefteq}

\def\semidirect{\rtimes}
\def\actson{\curvearrowright}

\def\class{\mathop{\mathrm{class}}\nolimits}
\def\exponent{\mathop{\mathrm{exponent}}\nolimits}
\def\rank{\mathop{\mathrm{rank}}\nolimits}
\def\Aut{\mathop{\mathrm{Aut}}\nolimits}

\def\Z{\mathbb{Z}}
\def\SL{\mathrm{SL}}

\def\Ab{\mathfrak{Ab}}
\def\olex{{\mathrm{olex}}}
\def\lex{{\mathrm{lex}}}

\def\A{\mathcal{A}}
\def\Ao{\A_o}
\def\Ar{\A_r}
\def\Ae{\A_e}
\def\Alex{\A_\lex}
\def\Aolex{\A_\olex}
\def\AOEzeta{\A_{O,E,\zeta}}
\def\IA{I_{\!\A}}
\def\Io{I_{o}}
\def\Ir{I_{r}}
\def\Ie{I_{e}}
\def\Ilex{I_\lex}
\def\Iolex{I_\olex}
\def\IOEzeta{I_{O,E,\zeta}}
\def\JA{J_{\!\A}}
\def\Jo{J_{o}}
\def\Jr{J_{r}}
\def\Je{J_{e}}
\def\Jlex{J_\lex}
\def\Jolex{J_\olex}
\def\ZJ{Z\kern-.1em J}

\def\ZJo{\ZJ_{o}}
\def\ZJr{\ZJ_{r}}

\def\ZJlex{\ZJ_\lex}
\def\ZJolex{\ZJ_\olex}
\def\OmegaZJ{\Omega\ZJ}

\def\OmegaZJr{\OmegaZJ_{r}}
\def\OmegaZJe{\OmegaZJ_{e}}

\begin{document}

\title{Variations on Glauberman's ZJ Theorem}
\author{Daniel Allcock}
\thanks{Supported by Simons Foundation 
Collaboration Grant 429818}
\address{Department of Mathematics\\University 
of Texas, Austin}
\email{allcock@math.utexas.edu}
\urladdr{http://www.math.utexas.edu/\textasciitilde allcock}
\date{April 3, 2021}

\begin{abstract}
    We give a new proof of Glauberman's ZJ Theorem,
    in a form that clarifies the choices involved
    and offers more choices than classical treatments.
    In particular, we introduce two new ZJ-type
    subgroups of a $p$-group~$S$, that contain
    $\ZJr(S)$ and $\ZJo(S)$ respectively and can be
    strictly larger.
\end{abstract}

\maketitle

\noindent
Glauberman's ZJ Theorem 
is a basic technical
tool in  finite group theory
\cite[Theorem~A]{Glauberman}.  
For example, it plays a major role
in the classification of simple groups having abelian
or dihedral Sylow $2$-subgroups.  
There are several versions of the theorem, 
depending on how one defines the Thompson subgroup.
We develop the theorem in a way that clarifies the 
choices involved, and offers more choices
than classical treatments.
In this paper all groups are taken to be finite.

Writing $S$ for a $p$-group,
the following are new.  First, 
Theorem~\ref{ThmAxiomaticZJ}
is an
``axiomatic'' version of the ZJ Theorem.
Second, we construct ZJ-type groups 
$\ZJlex(S)$ and $\ZJolex(S)$, which
contain $\ZJr(S)$
and $\ZJo(S)$ respectively, and can be strictly larger.
Third, we establish the ``normalizers grow'' property of
the Thompson-Glauberman replacement process,
and a consequence involving the Glauberman-Solomon
group $D^*(S)$; see
Theorems~\ref{ThmReplacement}\eqref{ItemReplacementNormalizersGrow}
and~\ref{ThmCentralizesDstar}.

\section{Introduction}
\label{SecIntro}

\noindent
Suppose $p$ is a prime,
$S$ is a $p$-group, $\Ab(S)$ is the set of
abelian subgroups of~$S$, and $\A\sset\Ab(S)$.
We set 
$$
\JA:=\gend{A:A\in\A}
\qquad
\IA:=\cap_{A\in\A}A
\qquad
\hbox{($\IA=1$ if $\A=\emptyset$).}
$$
$\JA$ is a sort of generalized Thompson subgroup,
and $\IA$ lies in its center.
For $P\subgp S$ we define $\A|_P$ as 
$\set{A\in\A}{A\subgp P}$
and $I_{\!\A|P}$ as $I_{\!\A|_P}$.
For other notation, and the definition of a $p$-stable action, see 
Section~\ref{SecPreliminaries}.

\begin{theorem}[``Axiomatic'' ZJ Theorem]
    \label{ThmAxiomaticZJ}
    Suppose $p$ is a prime, $S$ is a $p$-group and $G$ is a 
    group satisfying
    \begingroup
        \renewcommand{\theenumi}{\alph{enumi}}
        \begin{enumerate}
            \item
                \label{ItemSaSylow}
                $S$ is a Sylow $p$-subgroup of~$G$.
            \item
                \label{ItemZJFaithfulOnOp}
                $C_G(O_p(G))\subgp O_p(G)$.
            \item
                \label{ItemZJStability}
                $G$ acts $p$-stably on every 
                normal $p$-subgroup of~$G$.
        \end{enumerate}
    \endgroup
    \noindent   
    Then $\IA\normal G$ if $\A\sset\Ab(S)$
    has the following properties:
    \begin{description}
    \item[invariance (in~$S$, for~$G$)]
        $\forall P\normal S$, $I_{\!\A|P}$ is 
            $N_G(P)$-invariant.
    \item[replacement (in~$S$)]
        For every $B\normal S$ with class${}\leq2$,
        if there exist 
        members of $\A$ that contain $[B,B]$ but not~$B$,
        then
        $B$ normalizes one of them.
    \end{description}
    Furthermore, 
    $\IA$ is 
    characteristic in~$G$ if it is characteristic
    in~$S$.
\end{theorem}

Part of
the point of  ZJ-type theorems is to specify a subgroup
of~$S$ which will be characteristic in suitable~$G$, 
without referring to~$G$.  
We will say that a subgroup of $S$ has the
\emph{Glauberman property} (for~$S$)
if it is characteristic
in any group $G$ satisfying 
\eqref{ItemSaSylow}--\eqref{ItemZJStability}.
Replacing $N_G(P)$
with $\Aut(P)$ in the definition of invariance,
and quoting the theorem, lets us omit mention of~$G$:

\begin{corollary}
    \label{CorZJ}
    Suppose $p$ is a prime, $S$ is a $p$-group, and
    $\A\sset\Ab(S)$ satisfies {\bf replacement} (in~$S$)
    and also
    \begin{description}
        \item[full invariance (in~$S$)]
            $\forall P\normal S$, $I_{\!\A|P}$ is 
            $Aut(P)$-invariant.
    \end{description}
    Then $\IA$ has the Glauberman property.
    \qed
\end{corollary}

\noindent
These results allow $p=2$, but in this case 
no~$\A$ satisfying the conditions is known.
Also, every $\A$ we consider has the following property,
much stronger even than full invariance:
\begin{description}
    \item[completeness (in~$S$)]
        $\A$ contains every subgroup of~$S$ that
        is isomorphic to a member of~$\A$.
\end{description}

\begin{examples}
    The following subsets of $\Ab(S)$ are obviously
    complete in~$S$.  The first three are classical
    and the rest are new.
    We will abbreviate $I_{\!\A_{\cdots}(S)}$ and
    $J_{\!\A_{\cdots}(S)}$ to $I_{\cdots}(S)$ and $J_{\cdots}(S)$.
    \begin{align*}
        \Ao(S)
        &{}=
        \set{A\in\Ab(S)}
        {\hbox{$|A|\geq|A'|$ for all $A'\in\Ab(S)$}}
        \\
        \Ar(S)
        &{}=
        \set{A\in\Ab(S)}{\rank(A)=\rank(S)}
        \\
        \Ae(S)
        &{}=
        \set{A\in\Ar(S)}{\hbox{$A$ is elementary abelian}}
        \\
        \Alex(S)
        &{}=
        \set{A\in\Ab(S)}
        {\hbox{$A \geq_\lex A'$ for all $A'\in\Ab(S)$}}
        \\
        \Aolex(S)
        &{}=
        \set{A\in\Ao(S)}
        {\hbox{$A \geq_\lex A'$ for all $A'\in\Ao(S)$}}
        \\
        \AOEzeta(S)
        &{}=
        \set{A\in\Ab(S)}{\hbox{$|A|=p^O$, 
        $\exponent(A)\leq p^E$
        and $A\geq_\lex\zeta$}}
    \end{align*}
    In the last case, $O,E\in\Z$
    and $\zeta=(\zeta_1,\zeta_2,\dots)$ is a 
    sequence
    of integers.
    For sequences 
    $\zeta=(\zeta_1,\zeta_2,\dots)$
    and 
    $\zeta'=(\zeta_1',\zeta_2',\dots)$,
    $\zeta\geq_\lex\zeta'$
    refers to the usual lexicographic
    order.  
    When an abelian $p$-group $A$ appears on one side of $\geq_\lex$,
    the comparison refers to the sequence
    $$
    (\omega_1(A),\omega_2(A),\dots)
    :=
    \bigl(|\Omega_1(A)|,|\Omega_2(A)|,\dots\bigr).
    $$
    If $A,A'$ are abelian groups of the same
    order, then we think of $A>_\lex A'$ 
    as ``$A$ is closer to being
    elementary abelian than~$A'$ is.''  
\end{examples}

\begin{theorem}
    \label{ThmReplacementExamples}
    Suppose $p$ is an odd prime, $S$ is a $p$-group,
    and $D\normal S$. Then
    $\Ao(D)$, $\Ar(D)$, $\Ae(D)$, $\Alex(D)$,
    $\Aolex(D)$ and $\AOEzeta(D)$ (for any fixed
    $O,E,\zeta$) have the replacement property
    in~$S$.
\end{theorem}

\begin{corollary}
    \label{CorGlaubermanExamples}
    Every one of
    $\Io(S)$, $\Ir(S)$, $\Ie(S)$, $\Ilex(S)$,
    $\Iolex(S)$ and $\IOEzeta(S)$ has
    the Glauberman property for~$S$.
    \qed
\end{corollary}

Theorem~\ref{ThmReplacementExamples}
is a wrapper around
Glauberman's replacement theorem, extended 
to cover 
the last three cases (Theorem~\ref{ThmReplacement}).  
Corollary~\ref{CorGlaubermanExamples}
contains the classical
forms of the ZJ Theorem.  Namely: $\ZJo(S)$
and $\OmegaZJe(S)$ have the Glauberman
property.   This follows from
$$
\Io(S)=\ZJo(S)
\quad
\hbox{and}
\quad
\Ie(S)=\OmegaZJe(S).
$$
The first equality uses $\IA=Z(\JA)=C_S(\JA)$
when 
every member of
$\A$ is maximal in $\Ab(S)$ under
inclusion (Lemma~\ref{LemIisZJisCJ}). The second
is similar.  
From our perspective, the Glauberman property for $\ZJo(S)$
and $\OmegaZJe(S)$ derives from their coincidence with $I_o(S)$
and $I_e(S)$, and has nothing to do
with 
$\Jo(S)$ and
$\Je(S)$.
$\Ar(S)$ gives nothing new: 
Theorem~\ref{ThmOmegaZJrisOmegaZJe} shows
$$\Ir(S)=\Ie(S)=
\OmegaZJr(S)=\OmegaZJe(S).$$

We chose the new families
$\Alex(S)$ and $\Aolex(S)$ to
be ``small'', so that $J_{\cdots}(S)$ would also
be ``small'' and $I_{\cdots}(S)$ 
would be ``large''.
In particular,
\begin{align*}
\Ilex(S)
    =\ZJlex(S)=C_S(\Jlex(S))&{}\geq \ZJr(S)
    \\
\Iolex(S)
    =\ZJolex(S)=C_S(\Jolex(S))&{}\geq \ZJo(S).
\end{align*}
The equalities use
Lemma~\ref{LemIisZJisCJ}.
The containments follow from $\Alex(S)\sset\Ar(S)$
and $\Aolex(S)\sset\Ao(S)$, and
can easily be
strict (Examples \ref{EgBigZJlex} and~\ref{EgBigZJolex}). 
The containment $\ZJlex(S)\geq \ZJr(S)$
is interesting because 
it is not known whether $\ZJr(S)$ has the Glauberman property.
Any two
members of $\Alex(S)$ resp.\ $\Aolex(S)$
are isomorphic
to each other.  

There is no reason to expect $\AOEzeta(S)$
to be interesting; we include it mainly to give a sense
of what is possible using replacement.

Corollary~\ref{CorGlaubermanExamples} uses the $D=S$ case
of Theorem~\ref{ThmReplacementExamples}.  Since one can
take $D$ to be any normal subgroup there,
this
suggests trying to apply Theorem~\ref{ThmAxiomaticZJ}
to some
suitable
$\A\sset\Ab(D)$ with $D\normal S$.
In this way we can recover some recent results of 
K{\i}zmaz.
Recall that
$D\normal S$ is called \emph{strongly closed}
(in~$S$, with respect to $G\geq S$), if 
the only elements of $S$ which are $G$-conjugate
into~$D$ are the elements of~$D$.  
If this holds and 
$\A\sset\Ab(D)$ is complete (in~$D$),
then it is not hard to see that  
$\A$ satisfies invariance (in~$S$, for~$G$).
In fact strong closure is stronger than necessary for
this argument.

Therefore Theorem~\ref{ThmAxiomaticZJ} implies the following
``axiomatic'' version of \cite[Theorem~B]{Kizmaz}.
Corollary~\ref{CorKizmazExamples} 
below takes $D=\Omega_i(S)$,
and is our
analogue of \cite[Remark~1.4]{Kizmaz}.

\begin{theorem}
    \label{ThmAxiomaticKizmaz}
    Suppose $p$ is a prime, $S$ is a $p$-group,
    $G$ is a group satisfying
    \eqref{ItemSaSylow}--\eqref{ItemZJStability}
    of Theorem~\ref{ThmAxiomaticZJ}, and $D\normal S$ is
    strongly closed in~$S$ with respect to~$G$.
    Then $\IA\normal G$ for 
    any $\A\sset\Ab(D)$ which is complete
    (in~$D$) and satisfies replacement (in~$S$).
    \qed
\end{theorem}

\begin{corollary}
    \label{CorKizmazExamples}
    Suppose $p$ is an odd prime, $S$ is a $p$-group,
    $i\geq1$,
    and 
    $\Omega_i(S)$ has exponent${}\leq p^i$
    (for example, suppose $S$ has 
    class${}<p$).  Then  all of
    $\ZJo\Omega_i(S)$, 
    $\OmegaZJe\Omega_i(S)$, 
    $\ZJlex\Omega_i(S)$,
    $\ZJolex\Omega_i(S)$ and 
    $\IOEzeta(\Omega_i(S))$ (for any $O,E,\zeta$)
    have the Glauberman property for~$S$.
    \qed
\end{corollary}
    
I am grateful to Bernd Stellmacher 
and M. Yasir K{\i}zmaz for
helpful correspondence.

\section{Background and Notation}
\label{SecPreliminaries}

\noindent
We mostly follow the conventions of \cite{Gorenstein}.
Let $G$ be a group. 
If $w,x\in G$, then 
 $w^x$ means $x^{-1}wx$ and  $[w,x]$ means 
$w^{-1}x^{-1}wx$.
Brackets nest to the left, so $[x_1,\dots,x_n]$
means $[[x_1,\dots,x_{n-1}],x_n]$ when $n>2$.

Suppose $p$ is a prime.  If $S$ is a $p$-group,
then $\Omega_i(S)$ means the subgroup generated
by all elements of order${}\leq p^i$.  When $i=1$
we often write just $\Omega(S)$.
The rank of an abelian group means the size of the
smallest set of generators.
The rank of a nonabelian group means the maximum of
the ranks of its abelian subgroups.  We will
only use this notion for $p$-groups.
We sometimes suppress parentheses, eg
writing
$\OmegaZJe(S)$ for $\Omega(Z(\Je(S)))$.

The largest
normal $p$-subgroup of  $G$ is denoted $O_p(G)$.
Now suppose $G$ acts on a $p$-group~$P$.
We define $O_p(G\actson P)\normal G$ as the 
pre\-image of $O_p(G/C_G(P))$
under the natural map $G\to G/C_G(P)$.
This notation is nonstandard but natural; it
can be pronounced ``$O_p$ of $G$'s action on~$P$''.
We say that $x\in G$ acts quadratically
if $[P,x,x]=1$.  
The action of~$G$ on~$P$ is called $p$-stable if every element 
of~$G$ that acts quadratically lies in $O_p(G\actson P)$.
There is a simple ``global'' condition that guarantees
this: that 
no subquotient of~$G$ is isomorphic
to $\SL_2(p)$.  A proof of this 
can be extracted from that of
\cite[Lemma~6.3]{Glauberman}.
One main case of interest is when  $p$ is odd
and  $G$ has abelian
or dihedral Sylow $2$-subgroups. Having
quaternionic Sylow $2$-subgroups, $\SL_2(p)$
cannot arise as a subquotient.

We use the following elementary lemma several times.

\begin{lemma}
    \label{LemIisZJisCJ}
    Suppose $S$ is a $p$-group, $\A\sset\Ab(S)$,
    and every member of $\A$ is maximal in $\Ab(S)$
    under inclusion.  Then $\IA=Z(\JA)=C_S(\JA)$.
\end{lemma}

\begin{proof}
    The inclusions $\IA\subgp Z(\JA)\subgp C_S(\JA)$
    are obvious.  Now suppose
    $x\in C_S(\JA)$. For any $A\in\A$, $\gend{A,x}$
    is abelian, so the maximality of~$A$ 
    forces $x\in A$.
    Letting $A$ vary over $\A$ gives $x\in\IA$.
\end{proof}

\section{Replacement}
\label{SecReplacement}

\begin{theorem}[Replacement]
    \label{ThmReplacement}
    Suppose $p$ is a prime, 
    $S$ is a $p$-group and $B\normal S$.  
    If $p=2$ then assume $B$ is abelian.
    Suppose $A\subgp S$ is abelian and contains $[B,B]$.  

    Then either $B$ normalizes~$A$, or there exists
        $b\in N_B(N_S(A))-N_B(A)$.
        For any such~$b$,
    $A^*:=(A\cap A^b)[A,b]\subgp AA^b$
    enjoys the properties
    \begin{enumerate}
        \item
            \label{ItemReplacementSameOrder}
            $|A^*|=|A|$.
        \item
            \label{ItemReplacementPersistence}
            Like $A$, $A^*$ is abelian and contains $[B,B]$.
        \item
            \label{ItemReplacementIntersectionGrows}
            $A^*\cap B$ strictly contains $A\cap B$
            and is a proper subgroup of~$B$.
        \item
            \label{ItemReplacementNormalization}
            $A^*$ and $A$ normalize each other.
        \item
            \label{ItemReplacementNormalizersGrow}
             $N_S(A^*)$ contains $b$ and strictly
             contains $N_S(A)$.
        \item
            \label{ItemReplacementLex}
            \leavevmode
            If $p>2$, then
            $\omega_i(A^*)\geq\omega_i(A)$ 
            for all $i\geq1$.  
            In particular, 
            \begin{align*}
                \exponent(A^*)&{}\leq\exponent(A)
                \\
                \rank(A^*)&{}\geq\rank(A)
                \\
                A^*&{}\geq_\lex A.
                \qedhere
            \end{align*}
    \end{enumerate}
\end{theorem}

\noindent
Glauberman's replacement theorem
\cite[Theorem~4.1]{Glauberman}
adds the hypothesis that $\class(B)\leq2$, and
establishes \eqref{ItemReplacementSameOrder}--\eqref{ItemReplacementIntersectionGrows}.
This is enough to prove that
$\ZJo(S)$ has the Glauberman property.
Isaacs simplified the proof by
replacing some of the counting arguments with
structural ones \cite{Isaacs}.  He took $B$
abelian, as in Thompson's replacement theorem,
but with some work his arguments
can be adapted.  
Course notes
of Gagola \cite{Gagola} include a proof 
along these lines, 
citing
long-ago
unpublished work by (separately)
Isaacs, Passman and Goldschmidt.  
This includes the exponent inequality in
\eqref{ItemReplacementLex}, and removed Glauberman's
hypothesis on $\class(B)$.
K{\i}zmaz \cite{Kizmaz} independently adapted Isaacs' arguments from
\cite{Isaacs}
and proved his own generalization of Glauberman's replacement
theorem, namely \cite[Theorem~A]{Kizmaz}.
This includes
\eqref{ItemReplacementLex}, although he only stated the $i=1$ case
and the rank inequality.
He also clarified the overall argument by isolating the
commutator calculations in \cite[Lemma~2.1]{Kizmaz}, from which
our Lemma~\ref{LemReplacement} grew.

To my knowledge, \eqref{ItemReplacementNormalizersGrow}
is new.  It is
curious because it says that
$N_S(A)$
is a measure of how well-positioned
$A$ is with respect to~$B$, yet $N_S(A)$ is independent
of~$B$.  An interesting
consequence is that if 
$A\in\Ab(G)$ has largest possible normalizer, 
among all abelian subgroups of~$S$ with order~$|A|$,
then $A$ automatically centralizes the
ZJ-type
group $D^*(S)$ introduced by
Glauberman and Solomon \cite{GS}.  We postpone the details
until Theorem~\ref{ThmCentralizesDstar},
to avoid breaking the flow of
ideas.  

\begin{lemma}
    \label{LemReplacement}
    Suppose  a group $A$
    acts on a group $B$ 
     and centralizes
     $[B,B]$.  Then 
     the commutator subgroup of $[B,A]$
     is central in~$B$.  

     Furthermore, 
     if $A$ is abelian and $b\in B$ satisfies
     $[b,A,A,A]=1$, then the commutator subgroup of
     $[b,A]$  is an elementary 
     abelian $2$-group.
\end{lemma}

\begin{proof}
    (We do not use our blanket hypothesis that
    groups are finite, so $A$ and $B$ could be infinite.)
    Because $A$ centralizes $[B,B]$,
    so does $[B,A]$.  Two special cases of this are
    $[B,[B,A],[B,A]]=1=[[B,A],B,[B,A]]$.  Now the three
    subgroups lemma gives $[[B,A],[B,A],B]=1$, which
    is the first part of the lemma.

    The commutator subgroup of $[b,A]$ is abelian because it
    is central. It is generated by 
    the $[[b,x],[b,y]]$ with $x,y$ varying over~$A$.
    So it suffices to show that each has order${}\leq2$.
    We fix $x,y$ and abbreviate:
    \begin{gather*}
    b_x=[b,x]
    \quad
    b_y=[b,y]
    \\
    b_{x x}=[b,x,x]
    \quad
    b_{x y}=[b,x,y]
    \quad
    b_{y x}=[b,y,x]
    \quad 
    b_{y y}=[b,y,y].
    \end{gather*}
    By hypothesis, $x$ and $y$ centralize the last four of these.

    We will use the following identities, for any $u,v,w$ in any group:
    $$
    u^v=u[u,v]
    \qquad
    [u,vw]=[u,w][u,v]^w
    \qquad
    [uv,w]=[u,w]^v[v,w].
    $$
    In particular, $b^y=bb_y$ and $b_x^y=b_xb_{xy}$.
    Since $A$ centralizes
    $[B,B]$,
    $$
    [b_{xy},b]=[b_{xy}^y,b^y]
    =[b_{xy},bb_y]
    =[b_{xy},b_y][b_{xy},b]^{b_y\leftarrow \textrm{discard}}.
    $$
    We may discard the indicated conjugation because $[B,A]$
    centralizes $[B,B]$. 
    Canceling the $[b_{xy},b]$ terms leaves
    $1=[b_{xy},b_y]$.  Similarly,
    \begin{align*}
        [b_x,b]
        &{}=[b_x^y,b^y]=[b_x b_{xy},b b_y]
        =[b_x b_{xy},b_y]
        [b_x b_{xy}, b]^{b_y\leftarrow \textrm{discard}}
        \\
        &{}=[b_x,b_y]^{b_{xy}\leftarrow \textrm{discard}}  
        \,
        [b_{xy},b_y]
        \cdot
        [b_x,b]^{b_{xy}\leftarrow \textrm{discard}}
        \,
        [b_{xy},b].
    \end{align*}
    We discard conjugations as before, and
    we just saw that
    the second commutator is trivial. The first commutator
    is central,
    so we may cancel the $[b_x,b]$ terms.
    This leaves
    $(\star)$
    $1=[b_x,b_y][b_{xy},b]$.

    Next, we have
    $
    [b,xy]=[b,y][b,x]^y=b_y b_x^y=b_y b_x b_{xy}.
    $
    Exchanging $x$ and~$y$ doesn't change the left side, 
    so $b_y b_x b_{xy}
    =b_x b_y b_{yx}$.  Moving two terms to the right
    yields $b_{xy}=[b_x,b_y]b_{yx}$.  Bracketing 
    by $b$, and using the centrality of $[b_x,b_y]$, gives
    $[b_{xy},b]=[b_{yx},b]$.  By $(\star)$ and its analogue
    with $x$ and~$y$ swapped, this implies
    $[b_x,b_y]=[b_y,b_x]$. That is, $[b_x,b_y]^2=1$.
\end{proof}



\begin{proof}[Proof of Theorem~\ref{ThmReplacement}]
    Suppose $B$ does not normalize $A$.  
    Since $N_B(A)$ is proper in~$B$, it is proper in its
    own normalizer $N_B(N_B(A))$.  
    Because $N_S(A)$ normalizes
    $A$ and $B$, it also normalizes 
    $N_B(A)$ and $N_B(N_B(A))$,
    hence acts on  $N_B(N_B(A))/N_B(A)\neq1$.  
    So some
    $b\in N_B(N_B(A))-N_B(A)$ 
    is $N_S(A)$-invariant modulo~$N_B(A)$, ie
    $[b,N_S(A)]\subgp N_B(A)$.  This inclusion also says
    that
    $b$ normalizes~$N_S(A)$. 
    So $b\in N_B(N_S(A))-N_B(A)$, as claimed.

    Now set $N:=N_S(A)$ and suppose
    $b\in N_B(N)-N_B(A)$ is arbitrary.
    From $[B,B]\subgp A$ we have
    $A\cap B\normal B$, hence $A\cap B\subgp A\cap A^b$.

    From $A\normal N\normal\gend{N,b}$ 
    follows $A^b\normal N$.
    So $A,A^b$ normalize each other.  
    Setting $H=AA^b\normal N$, it follows that
    $[H,H]\subgp A\cap A^b\subgp Z(H)$.  
    In particular, $H$ has class${}\leq2$.
    The identity $(aa'^{-1})(a')^b=a[a',b]$, for  
    any $a,a'\in A$, shows that $H$ is also equal to
    $A[A,b]$.  

    Using bars for images in 
    $H/(A\cap A^b)$, obviously we have 
    $\Abar\cdot\overline{[A,b]}=\Hbar$.  On the other hand,
    $[A,b]$ lies in $H\cap B$, 
    and $\overline{H\cap B}$ meets $\Abar$
    trivially.  (Any element of $H\cap B$, that differs from an element
    of~$A$ by an element of $A\cap A^b$, lies in~$A$, hence
    in $B\cap A\subgp A\cap A^b$.)
    So $\overline{[A,b]}$ meets every coset of $\Abar$
    in $\Hbar$, yet lies in $\overline{H\cap B}$, which
    contains at most one point of each coset.
    Therefore
    $\overline{[A,b]}$ and $\overline{H\cap B}$ coincide
    and form
    a complement to $\Abar$ in~$\Hbar$.
    So
    $$
    A^*=(A\cap A^b)[A,b]=(A\cap A^b)(H\cap B)
    $$
    is a complement to $A$ in $H$, modulo~$A\cap A^b$.  
    Since $A^b$
    is another such complement, we have $A^*/(A\cap A^b)\iso
    A^b/(A\cap A^b)$ and therefore $|A^*|=|A|$,
    proving \eqref{ItemReplacementSameOrder}.

    \eqref{ItemReplacementPersistence} First,
   $[B,B]\subgp A\cap B\subgp A\cap A^b\subgp A^*$.  Now we
   prove $A^*$ abelian.  
    Because $A\cap A^b$ is central in~$H$ it is enough to
    prove
    $[A,b]$ abelian.
    If $p=2$ this follows from the
    hypothesis that $B$ is abelian.  So take $p$ odd.
    We may apply Lemma~\ref{LemReplacement} because
    $[B,B]\subgp A$ and 
    $
    [\,b,A,A,A]\subgp[H,A,A]\subgp[Z(H),A]=1.
    $
    The lemma
    shows that the commutator subgroup of 
    $[A,b]$ is a $2$-group,
    hence trivial.

    \eqref{ItemReplacementIntersectionGrows} 
    We already saw $A\cap B\subgp A\cap A^b\subgp A^*$.
    The strict containment $A\cap B< A^*\cap B$
    comes from
    the fact that $b$ does not normalize~$A$. Namely,
    $A$ omits 
    $b^{-1}ab$ for some $a\in A$, so it also omits
    $a^{-1}b^{-1}ab=[a,b]\in A^*\cap B$.
    And
    $A^*\cap B$ is strictly
    smaller than~$B$ because it lies 
    in $N$ and therefore omits~$b$.  

    \eqref{ItemReplacementNormalization} Both $A,A^*$ 
    contain $A\cap A^b$, hence $[H,H]$,  so 
    are normal in~$H$.

    \eqref{ItemReplacementNormalizersGrow} 
    $N$ 
    normalizes $A^*=(A\cap A^b)(H\cap B)$ because 
    it normalizes all four terms on the right.
    And $b$ normalizes $A^*=(A\cap A^b)[A,b]$ 
    because
    $$
    [A\cap A^b,b]\subgp[A,b]
    \quad\hbox{and}\quad
    [[A,b],b]\subgp[B,B]\subgp A\cap A^b.
    $$
    Because $b\notin N$ it follows that $N_S(A^*)$ is
    strictly larger than~$N$.


    \eqref{ItemReplacementLex} 
    We fix~$i$ and write $A_i$ for $\Omega_i(A)$.
    $H$ has class${}\leq2$, so the identities
    $$
    (xy)^e=x^ey^e[x,y]^{e(e-1)/2}
    \quad\hbox{and}\quad
    [x,y]^e=[x,y^e]
    $$ 
    hold for all $x,y\in H$.
    Together with
    the oddness of~$p$, 
    they show that $\Omega_i(H)$ has exponent${}\leq p^i$.
    Therefore its subgroup
    $$
    A_i^*:=(A_i\cap A_i^b)[A_i,b]\subgp A_iA_i^b\subgp \Omega_i(H)
    $$
    does too.   Now we reason as follows:
    $$
    \omega_i(A^*)\geq\omega_i(A_i^*)=|A_i^*|\geq|A_i|=\omega_i(A).
    $$
    The first step uses the obvious containment $A^*\geq A_i^*$,
    and the equalities use
    that $A_i^*$ and $A_i$ have exponent${}\leq p^i$.
    For the remaining inequality, observe that
    quotienting $A_iA_i^*=A_iA_i^b$ by~$A_i$ gives
    $A_i^*/(A_i^*\cap A_i)\iso A_i^b/(A_i^b\cap A_i)$.
    Because $A_i^*$ contains $A_i^b\cap A_i$ this
    implies $|A_i^*|\geq|A_i|$.  

    The rank and lex inequalities follow immediately.  
    By  $|A^*|=|A|$, the exponent inequality does too.
\end{proof}

\begin{proof}[Proof of Theorem~\ref{ThmReplacementExamples}]
    Write $\A$ for any one of 
    $\Ao(D),\dots,\AOEzeta(D)$.
    Supposing $B\normal S$,
    and that some $U\in\A$ contains $[B,B]$ but not~$B$,
    we will show that $B$ normalizes some
    $A\in\A$ that also contains
    $[B,B]$ but not~$B$.  Among all members of~$\A$
    that contain $[B,B]$ but not~$B$, and lie 
    in~$\gend{U^S}$, choose $A$ with $|A\cap B|$
    maximal.  Supposing that $B$ 
    does not normalize~$A$, we will derive a contradiction.

    Let $b$ and $A^*$ be as in Theorem~\ref{ThmReplacement}.
    In particular, $A^*$ is abelian and lies in
    $AA^b\subgp\gend{U^S}\subgp D$.
    So $A^*\in\Ab(D)$.
    $A^*$ contains $[B,B]$ 
    but not $B$ by \eqref{ItemReplacementPersistence} and 
    \eqref{ItemReplacementIntersectionGrows}.
    \eqref{ItemReplacementIntersectionGrows} also
    implies $|A^*\cap B|>|A\cap B|$, so
    the maximality in our choice of~$A$ forces
    $A^*\notin\A$.

    This is a contradiction because  $A^*\in\A$
    by other parts of Theorem~\ref{ThmReplacement}.  
    For $\Ao$ we use $|A^*|=|A|$.
    For $\Ar$ we use $\rank(A^*)\geq\rank(A)$.  
    For $\Ae$ we use both of these properties.
    For $\Alex$ we use $A^*\geq_\lex A$.
    For $\Aolex$ we use this
    and
    $|A^*|=|A|$. 
    For $\AOEzeta$ we use 
    $|A^*|=|A|$, $\exponent(A^*)\leq\exponent(A)$
    and $A^*\geq_\lex A$.
\end{proof}

In fact we have proven that~$\A$ satisfies a
strengthening of the replacement
axiom, got by removing ``with class${}\leq2$'' from the
statement of the axiom.  
We stated the axiom the way we did because
only the class${}\leq2$ case is needed to prove 
Theorem~\ref{ThmAxiomaticZJ}.

\section{Proof of the Axiomatic ZJ Theorem}
\label{SecZJ}

\begin{proof}[Proof of Theorem~\ref{ThmAxiomaticZJ}]
We write $I$ for $\IA$.
    We prove $I\normal G$ by induction 
    starting with $1\normal G$, by establishing the
    following
    inductive step:
    \begin{center}
        if 
        $\exists\, W\normal G$ with $W<I$,
        then
        $\exists\, B\normal G$ with $W<B\subgp I$.
    \end{center}
    Fix such a $W$.  
    Since $S$ preserves $I$, it
    acts on $I/W$.
     We define
    $X\subgp I$ as the preimage of the fixed-point
    subgroup.  So $X$ is normal in $S$,
    and is strictly larger than $W$
    because $I/W\neq1$.  
    To complete the proof we will show that
    $B:=\gend{X^G}\normal G$ lies in $I$.
    Supposing to the contrary, we will derive
    a contradiction.

    {\it Step 1: $B\subgp O_p(G)$.} 
    Being a subgroup of $I$, $X$ is abelian.
    Together with $X\normal S$ this gives
    $[O_p(G),X,X]\subgp[X,X]=1$.  
    Because $G$ acts $p$-stably
    on $O_p(G)$,
    $X$ lies in $O_p(G\actson O_p(G))$.  
    This equals $O_p(G)$
    because $C_G(O_p(G))$ is a $p$-group by hypothesis.  
    Since $X$ lies in $O_p(G)$,
    so does $B=\gend{X^G}$.

    {\it Step 2: $[B,B]\subgp W\subgp Z(B)$.}
    By the definition of $X$, $S$ acts trivially on $X/W$.
    In particular $O_p(G)$ does.  Conjugation shows that
    $O_p(G)$ acts trivially 
    (mod~$W$) on every 
    $G$-conjugate of $X$, hence trivially 
    on $B/W$. 
    That is, $[B,O_p(G)]\subgp W$.  By
    step~1 this implies
    $[B,B]\subgp W$.
    And $W\subgp Z(B)$ because $B$ is generated by
    abelian groups that contain~$W$.

    {\it Step 3: Set $H=O_p(G\actson B)$ and 
    $P=H\cap S$.  Then
    some $A\in\A|_P$ fails to contain~$B$.}
    Because we are supposing 
    $B\not\subgp I$, 
    some $A\in\A$ fails to contain~$B$.  It does
    contain~$[B,B]$, because step~2
    showed $[B,B]$ lies in~$W$, which lies in~$I$, hence~$A$.
    Step~2 also showed that $B\normal S$ 
    has class${}\leq2$.
    By the replacement property, 
    some member of~$\A$ contains $[B,B]$ but not~$B$,
    and is also
    normalized by~$B$.  
    We lose nothing by using it in place of~$A$,
    because it has all the properties of~$A$ established so far.
    That is, we 
    may suppose  $B$ normalizes~$A$.
    By
    $[B,A,A]\subgp[A,A]=1$ and 
    the $p$-stability of $G$'s action
    on $B$, we have
    $A\subgp H$.  
    Together with $A\subgp S$ this gives $A\subgp P$,
    hence $A\in\A|_P$.

    {\it Step 4: $B\subgp I_{\!\A|P}$.}  
    By the Frattini argument and the definition of~$H$,
    $$
    G=HN_G(P)=C_G(B)PN_G(P)\subgp C_G(X)N_G(P).
    $$
    So the $G$-conjugates of $X$ are the same as the 
    $N_G(P)$-conjugates.  It therefore suffices 
    to show that
    every $N_G(P)$-conjugate of $X$
    lies in~$I_{\!\A|P}$.  
    This follows from 
    $$
    X\subgp I
    ={\bigcap}_{A'\in\A}A'
    \ \mathop{\subgp}_{\lower10pt\hbox to 0pt{\hss\ $\begin{matrix}\uparrow\\
    \hbox{\smaller by $\A\supseteq\A|_P\neq\emptyset$}\end{matrix}$ 
        \hss}}\ 
    {\bigcap}_{A'\in\A|_P} A'
    = I_{\!\A|P}
    \mathop{\normal}_{\lower10pt\hbox to0pt{\hss$\begin{matrix}\uparrow\\
    \hbox{\smaller by invariance}\end{matrix}$\hss}}
    N_G(P).
    $$

    {\it The contradiction.} By $B\subgp I_{\!\A|P}$,
    every member of $\A|_P$ contains~$B$.  But in
    step~3 we found one which does not.

    \smallskip
    The final claim follows from the
    Frattini argument: $\Aut G$ is generated by 
    inner automorphisms and automorphisms that 
    preserve~$S$.
\end{proof}

Our method of ``growing'' the normal subgroup 
from $W$ to~$B$ derives from Stellmacher's
construction \cite[Theorem 9.4.4]{KS}\cite{Stellmacher}
of a different subgroup of~$S$ that also has
the Glauberman property.

\section{Etc}
\label{SecEtc}

\noindent
Here we collect some results and examples we mentioned 
in passing.  First, we claimed in the introduction
that our ZJ-type
groups $\ZJlex$ and $\ZJolex$ 
can be strictly larger than $\ZJr$ and $\ZJo$, respectively.

%

\begin{example}[$\ZJlex$ can be larger than $\ZJr$]
    \label{EgBigZJlex}
    Let $p$ be any odd prime.
    The group
    $$
    S=\biggend{x,y,u\bigm|
    1=[x,y]=x^{p^2}=y^p=u^p, x^u=xy, y^u=yx^p}
    $$
    is a semidirect product
    $(\Z/p^2\times\Z/p)\semidirect \Z/p$.
    (If $p=2$ then $S$ collapses to~$D_8$.)
    For any $a$ in $A:=\gend{x,y}$ but outside 
    $\gend{x^p}$, $C_S(a)=A$.
    It follows that $Z(S)$ can be no larger than
    $\gend{x^p}$.  
    Therefore $A$ is the unique
    abelian subgroup of~$S$ with order~$p^3$, because
    its intersection with any other such subgroup
    would be central in~$S$ and have order~$p^2$.
    In particular, $\rank S=2$,
    $\Alex(S)=\{A\}$ and $\ZJlex(S)=\Jlex(S)=A$.
    But $\Ar(S)$ also contains
    $\gend{x^p,u}\iso(\Z/p)^2$, so $\Jr(S)=S$ and 
    $\ZJr(S)=\gend{x^p}$.
\end{example}

\begin{example}[$\ZJolex$ can be larger than $\ZJo$]
    \label{EgBigZJolex}
    Let $p$ be any prime, and consider  
    the 
    ``Heisenberg group''
    $$
    S=\biggend{a,b,c\bigm| c=[a,b],\,1=[c,a]=[c,b]=a^{p^2}=b^{p^2}=c^{p^2}}.
    $$
    $\A_o(S)$ consists of the
    preimages of the $p^2+p+1$ order~$p^2$ subgroups
    of~$S/\gend{c}\iso(\Z/p^2)^2$.  So $\Jo(S)=S$ and $\ZJo(S)=\gend{c}$.
    One member of $\A_o(S)$ is isomorphic to $(\Z/p)^2\times\Z/p^2$, 
    namely $A:=\gend{a^p,b^p,c}$.  The rest are isomorphic
    to $(\Z/p^2)^2$, except if $p=2$, when there are also some
    isomorphic to $\Z/2\times\Z/8$.   So $A$ is the unique lex-maximal
    element of $\Ao(S)$, and
     $\Aolex(S)=\{A\}$ and
    $\Jolex(S)=\ZJolex(S)=A$.
\end{example}

In the introduction we mentioned $\OmegaZJr=\OmegaZJe$.
This is part of:

\begin{theorem}
    \label{ThmOmegaZJrisOmegaZJe}
    Suppose $p$ is a prime and $S$ is $p$-group.  Then
    $$
    \Ie(S)=\Ir(S)=\OmegaZJr(S)=\OmegaZJe(S)=\Omega C_S(\Jr(S))=\Omega C_S(\Je(S)).
    $$
\end{theorem}

\begin{proof}
    First, $\Ir(S)\subgp\Ie(S)$ because $\Ae(S)\sset\Ar(S)$.
    And $\Ie(S)\subgp\Ir(S)$ because every member of
    $\Ar(S)$ contains a member of~$\Ae(S)$, hence
    $\Ie(S)$.   We have proven the first equality.  For the others it
    is enough to establish the inclusions:
    $$
    \begin{tikzcd}[row sep=4ex,column sep=1.5em]
        \Ie(S) 
        \ar[r,equal]
        \ar[d,hook,"\mathrm{obvious}"]
        & \Ir(S)
        \ar[r,hook]
        & \OmegaZJr(S)
        \ar[rr,hook,"\mathrm{obvious}"]
        &&
        \Omega C_S(\Jr(S))
        \ar[d,hook,"\mathrm{by}\ \Je(S)\subgp\Jr(S)"]
        \\
        \OmegaZJe(S)
        \ar[rrrr,hook,"\mathrm{obvious}"]
        &&& 
        &\Omega C_S(\Je(S))
        \ar[r,hook]
        & \Ie(S).
    \end{tikzcd}
    $$
    The unlabeled inclusion in the top row 
is obvious, except for the fact that $\Ir(S)$ has exponent${}\leq p$,
    which holds by $\Ir(S)=\Ie(S)$.   
    The unlabeled inclusion in the bottom row is standard,
    with proof similar to that of Lemma~\ref{LemIisZJisCJ}.
\end{proof}

Just before Lemma~\ref{LemReplacement}, we mentioned that 
``normalizers grow'' during the Thompson-Glauberman replacement
process, and that this
forces abelian subgroups of~$S$ with ``large'' normalizers
to centralize $D^*(S)$.  Here $D^*(S)$ is the
characteristic subgroup introduced
by Glauberman and Solomon \cite{GS}, who
gave a lovely proof that it has the Glauberman
property.
Following Bender, 
$D^*(S)$ may be defined as the (unique)
largest normal subgroup of $S$ with the property that
it centralizes every abelian subgroup of~$S$ that it normalizes.
(It is easy to see that this exists.  And considering how
it acts, on abelian normal subgroups of itself, leads to a proof
that $D^*(S)$ is abelian.)  

\begin{theorem}
    \label{ThmCentralizesDstar}
    Suppose $p$ is a prime, $S$ is a $p$-group, and $A\in\Ab(S)$.
    Also suppose
    $N_S(A)$ is maximal under inclusion, among all
    groups $N_S(A^*)$ where $A^*\in\Ab(S)$ has the same order as
    $A$.
    Then $A$ centralizes $D^*(S)$.
\end{theorem}

\begin{proof}
    We apply our replacement theorem with $B$ equal to 
    the abelian group $D^*(S)\normal S$.  
    Arguing as for  Theorem~\ref{ThmReplacementExamples}
    shows that $D^*(S)$ normalizes~$A$.  
    So, by (Bender's)  definition, $D^*(S)$ centralizes~$A$.

    (If $p>2$ then this argument gives a slightly stronger
    result, with $A^*$ varying over fewer elements of $\Ab(S)$,
    namely those that
    satisfy $|A^*|=|A|$
    and also
            $\omega_i(A^*)\geq\omega_i(A)$ for all~$i$.)
\end{proof}

%

\end{document}